\documentclass[a4paper, 12pt]{article}
\usepackage{amsfonts}
\usepackage {amssymb}
\usepackage {amsmath}
\usepackage {amsmath}
\usepackage {amsthm}
\usepackage{graphicx}
\usepackage{multirow}
\usepackage{xypic}
\usepackage {amscd}
\usepackage{mathrsfs}
\usepackage[colorlinks, linkcolor=blue, anchorcolor=black, citecolor=red]{hyperref}
\usepackage{enumerate}

\usepackage{geometry}  

\newcommand{\ord}{{\rm ord}}

\newcommand{\vphi}{\varphi}

\newcommand{\bC}{{\mathbb{C}}}

\newcommand{\bP}{{\mathbb{P}}}

\newcommand{\bQ}{{\mathbb{Q}}}

\newcommand{\bR}{{\mathbb{R}}}

\newcommand{\mcO}{{\mathcal{O}}}

\newcommand{\bZ}{{\mathbb{Z}}}

\newcommand{\bN}{{\mathbb{N}}}

\newcommand{\la}{\langle}
\newcommand{\ra}{\rangle}

\newcommand{\PSH}{{\rm PSH}}

\newcommand{\bB}{\mathbb{B}}

\newcommand{\cO}{\mathcal{O}}

\newcommand{\mult}{{\rm mult}}

\newcommand{\fa}{\mathfrak{a}}
\newcommand{\cE}{\mathcal{E}}

\newcommand{\cJ}{\mathcal{J}}

\newcommand{\ddc}{{\rm dd^c}}

\newcommand{\cC}{\mathcal{C}}
\newcommand{\bo}{{\bf 1}}
\newcommand{\cU}{\mathcal{U}}
\newcommand{\cI}{\mathcal{I}}
\newcommand{\fm}{\mathfrak{m}}

\newtheorem{thm}{Theorem}[section]
\newtheorem{prop}[thm]{Proposition}

\newtheorem{cor}[thm]{Corollary}
\newtheorem{rem}[thm]{Remark}
\newtheorem{conj}[thm]{Conjecture}

\newtheorem{exmp}[thm]{Example}
\newtheorem{lem}[thm]{Lemma}

\newtheorem{defn-prop}[thm]{Definition-Proposition}

\begin{document}

\title{Analytical approximations and Monge-Amp\`{e}re masses of plurisubharmonic singularities}
\author{Chi Li}
\date{}

\maketitle

\abstract{
We construct examples of plurisubharmonic functions with isolated singularities at $0\in \bC^n$, whose residual Monge-Amp\`{e}re masses at the origin can not be approximated by masses of canonical analytic approximations obtained via multiplier ideals. This answers negatively a conjecture of Demailly, and shows that residual Monge-Amp\`{e}re masses are not valuative invariants of plurisubharmonic singularities.
}

\tableofcontents

\section{Introduction and main results}

Let $\vphi$ be a plurisubharmonic function in a bounded pseudoconvex open neighborhood $\Omega$ of $0\in \bC^n$. Assume that $\vphi$ is locally bounded on $\Omega^*=\Omega\setminus \{0\}$. For any $\lambda>0$, we have the multiplier ideal sheaf $\cJ(\lambda\vphi)$: for any open set $U$ in the domain of $\vphi$, $\cJ(\lambda\vphi)(U)$ is generated by the space of holomorphic functions $f$ such that 
$$
\int_{U}|f|^2 e^{-\lambda\vphi}dV_n<+\infty
$$
where $dV_n$ is the standard volume form on $\bC^n$. 
Let $\{f^{(\lambda)}_k\}_{k\ge 1}$ be an orthonormal basis of $\cJ(\lambda\vphi)(\Omega)$ and set:
\begin{equation}
\vphi_\lambda=\frac{1}{\lambda}\log \sum_k |f^{(\lambda)}_k|^2.
\end{equation}
Demailly used Ohsawa-Takegoshi extension theorem to prove that the singularity of $\vphi$ can be approximated by analytic singularities in an accurate manner:
\begin{thm}[{Demailly, see \cite[Theorem 7]{Dem15}}]
With the above notations, there exist constants $C_1, C_2>0$ such that for $m\in \bZ_{>0}$, 
\begin{enumerate}
\item[(a)] $\vphi(z)-\frac{C_1}{m}\le \vphi_m(z)\le \sup_{|\zeta-z|<r}\vphi(\xi)+\frac{1}{m}\log \frac{C_2}{r^n}$ for every $z\in \Omega$ and $r< d(z, \partial \Omega)$. In particular, $\vphi_m$ converges to $\vphi$ pointwise and in $L^1_{\rm loc}$ topology on $\Omega$ as $m\rightarrow +\infty$, and
\item[(b)] $e_1(\vphi, z)-\frac{n}{m}\le e_1(\vphi_m, z)\le e_1(\vphi, z)$, where $e_1(\vphi, z)=\liminf_{w\rightarrow z}\frac{\vphi(w)}{\log|w-z|}$ is the standard Lelong number of $\vphi$ at $z$.
\end{enumerate}
\end{thm}
When the unbounded locus of $\vphi$ is a compact subset of a Stein manifold, Demailly showed that the Monge-Amp\`{e}re operator $\vphi\mapsto (\ddc\vphi)^n$ is well-defined and is continuous with respect to decreasing sequences of psh functions (see \cite[III.(4.3)]{Dembook}). Define the residual Monge-Amp\`{e}re mass of $\vphi$ at $0\in \bC^n$ as:
\begin{equation}
e_n(\vphi)=e_n(\vphi,0)=\int_{\{0\}}(\ddc\vphi)^n. 
\end{equation}
This is the generalized Lelong number of the current $(\ddc\vphi)^n$ at $0\in \bC^n$ (see \cite[III.\S 5]{Dembook}).  
Motivated by the above result, Demailly conjectured the following:
\begin{conj}[{Demailly, \cite{DH12}, \cite[Problem 8]{DGZ16}}]\label{conj-Dem}
As $m\rightarrow +\infty$, $e_n(\vphi_m)$ converges to $e_n(\vphi)$. 
\end{conj}
%There have been several attempt to prove this. See \cite{KR20, Ras13} and references therein for positive results in this direction for special classes of psh singularities.
In a closely related development, in \cite{BFJ08} Boucksom-Favre-Jonsson carried out a deep study of plurisubharmonic singularities using the tool of valuations. More precisely, for any plurisubharmonic germ $\vphi$ defined near $0\in \bC^n$ and any divisorial valuation $v$ centered at $0$, the authors defined a generalized Lelong number, which we simply denoted by $v(\vphi)$. We state a particular consequence of the main result from \cite{BFJ08}: 
\begin{thm}[\cite{BFJ08}]\label{thm-BFJ}
Let $\vphi, \psi: (\bC^n, 0)\rightarrow \bR\cup \{-\infty\}$ be two psh germs with isolated singularity at $0\in \bC^n$. The following two statements are equivalent:
\begin{enumerate}
\item $v(\vphi)=v(\psi)$ for every divisorial valuation centered at $0$ (In other words, $\vphi$ and $\psi$ are valuatively equivalent). 
\item $\cJ(\lambda\vphi)=\cJ(\lambda\psi)$ for any $\lambda>0$. 
\end{enumerate}
\end{thm}
In \cite{KR20}, Kim-Rashkovskii proposed the following
\begin{conj}[\cite{KR20}]\label{conj-KR}
Let $\vphi$ and $\psi$ be psh functions with isolated singularities at $0\in \bC^n$. If $v(\vphi)=v(\psi)$ for every divisorial valuation $v$  centered at $0$, then $e_n(\vphi, 0)=e_n(\psi, 0)$. 
\end{conj}
In other words, this conjecture says that the residual Monge-Amp\`{e}re masses are valuative invariants of psh singularities. 
By Theorem \ref{thm-BFJ}, it is easy to see that Conjecture \ref{conj-Dem} implies Conjecture \ref{conj-KR}. These two conjectures were known to be true for special classes of psh singularities, including analytic psh singularities or a more general class of tame psh singularities  (see \cite{BFJ08, KR20, Ras13} for details of positive results).

However, we will construct counterexamples to show that neither of these two conjectures is true in general. To construct such  an example, we 
consider the projection $\pi: X^*:=\bC^2\setminus\{0\}\rightarrow \bP^1=(\bC^2-\{0\})/\bC^*$. Then $X^*$ can be identified with the complement of the zero section in the total space of the the tautological line bundle $\cO(-1)\rightarrow \bP^1$. The function $h_0(z)=|z|^2$ on $\bC^2$ can be identified with the standard Hermitian metric on $\cO(-1)$ whose Chern curvature is give by $-\partial\bar{\partial}\log h_0=-\frac{2\pi}{\sqrt{-1}} \omega_0$ where  $\omega_0=\frac{\sqrt{-1}}{2\pi}\frac{dw\wedge d\bar{w}}{(1+|w|^2)^2}$ with respect to the standard coordinate on $\bC= \bP^1\setminus \{\infty\}$. Moreover any psh Hermitian metric on $\cO(-1)$, which may be singular, is of the form $h=h_0 e^{u}$ where $u$ is a $\omega_0$-psh function on $\bP^1$, i.e. $\omega_0+\ddc u\ge 0$ where we set $\ddc=\frac{\sqrt{-1}}{2\pi}\partial\bar{\partial}$. $h$ can be considered as a function on $\bC^2$ by the identity $h=|z|^2 e^{\pi^*u}$. Moreover the function $\log h=\log |z|^2+\pi^*u$ is a psh function on $\bC^2$. 
Note that although $\pi^*u$ is not defined at $0\in \bC^2$, $h(0)=0$ and $\log h(0)=-\infty$ are both well-defined. 

With the above data, we get a psh function on $\bC^2$ (and hence on any pseudoconvex domain containing $0\in \bC^2$) by setting:
\begin{equation}\label{eq-vphi}
\vphi=\max\{\log h, 2 \log |z|^2\}=\max\{ \log|z|^2+\pi^*u, 2 \log |z|^2\}.
\end{equation}
Because $\vphi\ge 2 \log |z|^2$, $\vphi$ has isolated singularity at $0$, i.e. locally bounded on $\Omega\setminus\{0\}$. In classical potential theory, there is an example of subharmonic function
on $\bC$ whose associated Riesz measure does not have atoms and is supported on a generalized Cantor set which is polar. Similarly one can construct such a $\omega_0$-psh function $u$ on $\bP^1$ (see section \ref{sec-potential}). We will use such an $\omega_0$-psh function to disprove both Conjecture \ref{conj-Dem} and Conjecture \ref{conj-KR}:
\begin{thm}\label{thm-main}
There exists an $\omega_0$-psh function $u$ on $\bP^1$ such that the corresponding $\vphi$ defined in \eqref{eq-vphi} satisfies
\begin{enumerate}
\item $\cJ(\lambda \vphi)=\cJ(\lambda \log|z|^2)$ for any $\lambda>0$. 
\item 
 $e_2(\vphi)=2$ and $e_2(\vphi_m)=(\frac{m-1}{m})^2$ for all $m\in \bZ_{>0}$.
 \end{enumerate}
\end{thm}
This example also induces higher dimensional counterexamples on $\bC^n$. Moreover, we can construct examples which are maximal in punctured neighborhoods of $0\in \bC^n$. See section \ref{sec-hd} for details.

%The construction of $u$ depends on a standard construction from potential theory. 
Our example is motivated by an analogous global example in the study of geodesic rays in the space of finite energy metrics in \cite[Example 6.10]{BBJ18} (based on a construction from \cite{Dar17}). Indeed, ruling out a similar kind of phenomenon for destabilizing geodesic rays in the constant scalar curvature K\"{a}hler (cscK) problem is a key step in the author's recent work in \cite{Li20} proving an existence result towards the Yau-Tian-Donaldson conjecture: the K-stability for models (or filtrations associated to models) is a sufficient condition for the existence of cscK metric on any polarized K\"{a}hler manifold. 

While being a natural question, Demailly actually proposed his conjecture to study Guedj-Rashkovskii's conjecture (see \cite[Problem 7]{DGZ16}) which asks whether $e_1(\vphi)=0$ implies $e_n(\vphi)=0$. The construction in this paper only provides examples of psh functions with positive Lelong numbers and leaves open the latter conjecture.

{\bf Acknowledgement:} 
The author is partially supported by NSF (Grant No. DMS-1810867) and an Alfred P. Sloan research fellowship. 
He would like to thank Professor J.-P. Demailly for his interest and comments.

%The strategy now is to choose $h$ or equivalently the function $u$ carefully to achieve our wanted property. 

\section{Multiplier ideal approximation}\label{sec-mult}

We first study the analytic approximation $\vphi_m$. We fix a pseudoconvex domain $\Omega\subset \bC^2$ that is $S^1$-invariant, i.e. $(e^{i\theta}z_1, e^{i\theta} z_2)\in \Omega$ for any $e^{i\theta}\in S^1$ and $(z_1, z_2)\in \Omega$. For example, we can just let $\Omega$ to be the unit ball $\bB_1$ in $\bC^2$. First note that $\vphi$ is $S^1$-invariant. So $\cJ(\lambda\vphi)(\Omega)$ is $S^1$-invariant. By using the Taylor expansion of holomorphic functions, we see that which implies that $\cJ(\lambda\vphi)$ is homogeneous, i.e. $\cJ(\lambda\vphi)=\bigoplus_{k} \cJ(\lambda\vphi)_{k}$ where $\cJ(\lambda\vphi)_{k}$ consists of homogeneous holomorphic functions in $\cJ(\lambda\vphi)$ of degree $k$. In other words, any $f\in \cJ(\lambda\vphi)_{k}$ satisfies $f(t z)=t^k f(z)$. It is easy to see that $f$ must be a homogeneous polynomial of degree $k$. Each such homogeneous polynomial corresponds to a holomorphic section $s_f\in H^0(\bP^1, \mcO(k))$. 
\begin{prop}
Let $u$ be a $\omega_0$-psh function and $\vphi$ be defined as in \eqref{eq-vphi}. The following statements are true:
\begin{enumerate}
\item If $\fm$ denotes the maximal ideal of $0\in \bC^2$, then
$\cJ(\lambda\vphi)(\Omega)\subseteq \fm^{\lceil \lambda\rceil-1}$.
\item
If $k\ge \lceil \lambda \rceil-1$ and $s_f\in H^0(\bP^1, \mcO(k)\otimes \cJ(h^{-\lambda}))$ where $\cJ(h^{-\lambda})$ is the multiplier ideal sheaf of the Hermitian metric $h^{-\lambda}$, then $f\in \cJ(\lambda\vphi)(\Omega)$.
\end{enumerate}
\end{prop}
\begin{proof}
Note that $\vphi=\log |z|^2+\max\{\pi^*u, \log |z|^2\} \le \log |z|^2+C$ for some constant $C>0$. So $\cJ(\lambda\log|z|^2)\supseteq \cJ(\lambda\vphi)$. On the other hand, it is straightforward to verify that $\cJ(\lambda\log|z|^2)(\Omega)$ consists of all homogeneous polynomial of degree at least $\lceil \lambda \rceil-1$. 

Next we verify the second statement. Let $\rho: X\rightarrow \bC^2$ be the standard blowup at $0\in \bC^2$. Denote the exceptional divisor by $D$. Then $X$ is also the total space of the tautological line bundle $\pi: \cO_{\bP^1}(-1)\rightarrow \bP^1$ and can be covered by two affine coordinate charts. Consider one such coordinate chart, for example $\{\cU_2=\pi^{-1}(U_2); w=\frac{z_1}{z_2}, z_2\}$ where $U_2=\{z_2\neq 0\}\subset \bP^1$. Then $\rho^* f=f(w z_2, z_2)=z_2^k f(w, 1)$ and
\begin{eqnarray*}
\rho^*\vphi&=&\max\{\log |z_2|^2+\log (1+|w|^2)+u(w), 2 \log |z_2|^2+2 \log(1+|w|^2)\}\\
&=&\log |z_2|^2+\vphi'(w, z_2)
\end{eqnarray*}
where $\vphi'(w, z_2)=\max\{\psi(w), \log |z_2|^2+2\log(1+|w|^2)\}$
and $\psi(w)=\log(1+|w|^2)+u(w)$ is the local potential of $h=h_0e^u$ over $U_2$.
Note that $\rho^*dV_2=|z_2|^{2(n-1)} dV_2$ with $n=2$. Because $\vphi'(w, z_2)\ge \psi(w)$, we get:
\begin{eqnarray*}
\int_{\cU_2} |f|^2 e^{-\lambda\vphi}d\lambda&=&\int_{\cU_2} |f(w,1)|^2 |z_2|^{2k}|z_2|^{-2\lambda}e^{-\lambda\vphi'(w, z_2)}|z_2|^{2(n-1)}dV\\
&\le& \int_{\cU_2} |f(w, 1)|^2 e^{-\lambda\psi(w)} |z_2|^{2(k-\lambda+n-1)}\sqrt{-1}^2 dw\wedge  d\bar{w}\wedge dz\wedge d\bar{z}.
\end{eqnarray*}
The last integral is finite if and only if $f(w,1)\in \cJ(\lambda\psi)(U_2)$ and $2(k-\lambda+n-1)+1>-1$. The same argument works for the other coordinate chart. So the statement follows. 
\end{proof}
\begin{cor}\label{cor-Lelong}
For any $\lambda>0$, if $\cJ(h^{-\lambda})=\mcO_{\bP^1}$, then $\cJ(\lambda\vphi)(\Omega)=\{f\in \mcO(\Omega); \ord_0(f)\ge \lceil \lambda\rceil-1\}=\cJ(\lambda \log|z|^2)(\Omega)$. In this case $e_2(\vphi_\lambda)=\left(\frac{\lceil \lambda\rceil-1}{\lambda}\right)^2$. In particular, this holds true if the local potential of $h^{-1}$ has zero Lelong numbers.
\end{cor}
\begin{proof}
By Skoda's estimate, if the local potential of the psh metric $h^{-1}$ has zero Lelong number, then $\cJ(h^{-\lambda})=\cO_{\bP^1}$ for any $\lambda>0$. So the statement follows easily from the above proposition.
\end{proof}

\section{A recap on potential theory}\label{sec-potential}

We recall some standard potential theory on $\bP^1$. Let $\omega_0=\frac{\sqrt{-1}}{2\pi}\frac{dw\wedge d\bar{w}}{(1+|w|^2)^2}$ be the round metric with volume 1 on $S^2$. Let $u$ be any $\omega_0$-psh function. With $\Delta u=\Delta_{\omega_0}u=\frac{\sqrt{-1}\partial\bar{\partial}u}{\omega_0}$ we can write:
\begin{equation}\label{eq-meas}
\omega_0+\frac{\sqrt{-1}}{2\pi}\partial\bar{\partial}u=\omega_0(1+\frac{1}{2\pi}\Delta u)
\end{equation}
is a positive closed $(1,1)$-current and we call the quantity in \eqref{eq-meas} the (Riesz) probability measure associated to $u$. Note that if $u\in C^2(\bP^1)$, then $\Delta u=(2\pi)(1+|w|^2)^2 u_{w\bar{w}}$ with respect to the standard coordinate chart  $(\bC, w)$. 

Conversely, we get a $\omega_0$-psh function for any probability measure on $\bP^1$. This is obtained by using the 
Green function $G(z, w)$ of the Laplace operator of $\omega_0$. $G(z, w)$ satisfies the following equation:
\begin{equation}\label{eq-Green}
(-\Delta) G(z, w)=\delta_0(z)-1, \quad G(z, w) \text{ is locally bounded on } \bC^2\setminus \{z=w\}.
\end{equation}
The following lemma is well known. We give the short proof for the reader's convenience.
\begin{lem}
The solution to the equation \eqref{eq-Green} is given by: 
\begin{equation}
G(z, w)=-\frac{1}{\pi}\log\frac{|z-w|}{\sqrt{1+|z|^2}\sqrt{1+|w|^2}}.
\end{equation}
\end{lem}
\begin{rem}
The expression $\frac{|z-w|}{\sqrt{1+|z|^2}\sqrt{1+|w|^2}}$ is nothing but the chordal distance between $z$ and $w$ on the Riemann sphere under stereographic projection.
\end{rem}
\begin{proof}
First assume that $z=0$. Then by rotational symmetry, $G(0, w)=G(|w|)$. Let $r=|w|$. Then for $r>0$, $G(r)$ satisfies the equation:
$$
(2\pi)(1+r^2)^2\frac{1}{4}(G''+r^{-1}G')=1.
$$
The general solution is given by $G(r)=\frac{1}{2\pi}\left(\log(r^2+1)-C\log r^2\right)$. Since $G(r)$ is bounded as $r\rightarrow+\infty$, we get $C=1$, i.e. $G(r)=\frac{1}{2\pi}\log \frac{1+r^2}{r^2}$. For general $z$, we consider the unitary M\"{o}bius transformation: $w\mapsto \frac{w-z}{1+\bar{z}w}$ which preserves the metric. So we get:
\begin{equation*}
G(z, w)=G\left(0, \frac{w-z}{1+\bar{z}w}\right)=\frac{1}{2\pi}\log\frac{(1+|w|^2)(1+|z|^2)}{|w-z|^2}.
\end{equation*}

\end{proof}
For any probability measure $d\mu$ on $\bP^1$, we define the potential of $\mu$ to be:
\begin{equation}\label{eq-potential}
p_\mu(z)=2\pi \int_{\bP^1}(-G(z, w))d\mu(w).
\end{equation}
Then $p_\mu(z)$ is a $\omega_0$-psh function satisfying $\omega_0+\frac{\sqrt{-1}}{2\pi}\partial\bar{\partial} p_\mu=\mu$.
 
Recall that a subset $E\subset\bC$ is polar if there exists an $\omega_0$-psh function $u$ such that $E\subseteq\{u=-\infty\}$. 
%By the standard potential theory, a subset $E$ is polar if and only if its capacity, denoted by $c(E)$, is equal to 0. We refer to \cite[]{} for the details. 
It is well-known that there exist subharmonic functions on $\bC$ whose Riesz measures do not charge isolated points but charge polar sets. \footnote{The author learned the existence of such an example from \cite[Example 6.10]{BBJ18} which was also pointed out in \cite[paragraph after Corollary 1.8]{GZ07}.} The corresponding polar sets in these examples are generalized Cantor sets, whose construction we now recall following \cite[5.3]{Ran95} (see also the earlier reference \cite[Theorem 3]{Car67}). Let ${\bf s}=\{s_k\}_{k\ge 1}$ be a sequence of numbers such that $0<s_k<1$ for all $k$. Define $C(s_1)$ to be the set obtained from $[0, 1]$ by removing an open interval of length $s_1$ from the center. At the $k$-th stage, let $C(s_1, \dots, s_k)$ be the set obtained by removing from the middle of each interval in $C(s_1, \dots, s_{k-1})$ an open subinterval whose length is a proportion $s_k$ of the whole interval. In this way we get a decreasing sequence of compact sets $(C(s_1, \dots, s_k))_{k\ge 1}$, and the associated {\it generalized Cantor set} is defined to be:
\begin{equation}
\cC:=C({\bf s})=\bigcap_{m\ge 1}C(s_1,\dots, s_k).
\end{equation} 
It can be checked that $\cC$ is a compact, totally disconnected set of Lebesgue measure $\prod_{k=1}^{+\infty}(1-s_k)$. By estimating the capacity of the generalized Cantor set, it is shown in \cite[Theorem 5.3.7]{Ran95}, if $1-s_k$ decays very fast, then $\cC$ is polar. 

On the other hand, by (advanced) real analysis, we can associate to any such generalized Cantor set a generalized Cantor function $f(x)$ defined on $[0, 1]$. For any $k\ge 1$, $C(s_1, \dots, s_k)$ consists of $2^k$ pieces and each piece has length $l_k:=\prod_{j=1}^k \frac{1-s_j}{2}$. For convenience, we also set $ \quad l_0=1$. The generalized Cantor function $c(x)$ can be defined in the following way:
\begin{equation}
\renewcommand{\arraystretch}{1.5}
c(x)=\left\{
\begin{array}{ll}
\sum_{k=1}^{+\infty} \frac{b_k}{2^k} & x=\sum_{k=1}^{+\infty}  \frac{b_k(1+s_k)}{2} l_{k-1} \text{ for } b_k\in \{0, 1\} \\
\sup_{y\le x, y\in \mathcal{C}} c(y) & x\in [0,1]\setminus \cC.
\end{array}
\right.
\end{equation}
The generalized Cantor function $c(x)$ is non-decreasing and takes every value between $[0, 1]$. Hence it is continuous. It is the distribution function associated to a probability measure $\mu$ with support ${\rm supp}(\mu)=\cC$. The continuity of $c(x)$ means that $\mu$ has no atoms. We assume that $s_k>1/3$. 
Fix any $x\in \cC$. Then for any $k\ge 1$, the mass of the subset $[x-l_k, x+l_k]\cap \cC$ is equal to $2^{-k}$. 
%\begin{thm}[{\cite[Theorem 5.3.7]{Ran95}}]
%Its capacity $c(C(s))$ is can be estimated as follows:
%\begin{equation*}
%\frac{pq}{2}\le c(C({\bf s}))\le \frac{p}{2},
%\end{equation*}
%where $p=\prod_{n\ge 1}(1-s_k)^{1/2^n}$ and $q=\prod_{n\ge 1}s_k^{1/2^n}$.
%\end{thm}
%As a consequence, if $1-s_k$ decreases very fast such that $p=0$, then the corresponding generalized Cantor set $C({\bf s})$ is polar. 

%From now on we choose $1-s_k=a^{-n}$ for some $a>3/2$ to be determined later. 
Let $p_\mu$ be the potential of $\mu$, as given by formula \eqref{eq-potential}. Then we know that $p_\mu$ is $\omega_0$-harmonic outside $\cC$. In other words, $p_\mu$ is smooth on $\bP^1\setminus \cC$ and satisfies $\omega_0+\ddc p_\mu=0$. 

From now on set $l_k=e^{-a^k}$ for $k\ge 1$ and define $s_k$ inductively by the formula:
\begin{equation}\label{eq-sn} 
s_k=1-\frac{2l_k}{l_{k-1}}=1-2 e^{a^{k-1}-a^k}.
\end{equation}
It is easy to see that if $a>2$, then $s_k\in (0,1)$. We will need the following refined information for $p_\mu$ which also directly proves that $\cC$ is polar in this case.
\begin{prop}\label{prop-potential}
With the above choices of $s_k$ in \eqref{eq-sn} with $a>2$, the following statements are true:
\begin{enumerate}
\item[(i)]
$\cC=\{w\in \bP^1; p_\mu(w)=-\infty\}$. In particular, $\cC$ is polar.
\item[(ii)] $e^{p_\mu}$ is a continuous function on $\bP^1$.
\item[(iii)] There exists $C>0$ such that $p_\mu(w)\ge 2 \log {\rm dist}_{\rm sph}(w,\cC)-C$, where ${\rm dist}_{\rm sph}(\cdot, \cC)$ is the spherical distance to the closed set $\cC$.
\end{enumerate}
\end{prop}
\begin{proof}
Note that we always have $-G(z, w)\le \frac{1}{2\pi}\log |z-w|^2$. Let $d=d(z)={\rm dist}_{\rm Euc}(z, \cC)$ denote the Euclidean distance from $z$ to $\cC$. For any $n\ge 1$, we have the estimate:
\begin{equation}
p_\mu(z)\le 2^{-k} \log (d^2+l_k^2)+(1-2^{-k}) \log(d^2+1).
\end{equation}
In particular, when $0\le d<l_k\le 1$, then there there exists a constant $C>0$ independent of $k$ such that
\begin{equation*}
p_\mu(z)\le 2^{-k}\cdot 2 \log l_k+C=-2 (\frac{a}{2})^k+C.
\end{equation*}
Since $a>2$, we let $k\rightarrow +\infty$ to conclude that $p_\mu(z)=-\infty$ if $z\in \cC$ and $p_\mu(z)\rightarrow -\infty$ uniformly as $d(z)\rightarrow 0$. As mentioned above, $p_\mu$ is smooth outside $\cC$. So we get (i) and (ii).

Moreover, fixing any $b>2$, for any $w\in \cC$, we have
\begin{eqnarray*}
-(2\pi) G(z, w)&\ge&  \log |z-w|^2-\log 2-\log (1+(b+1)^2) \quad \text{ if } d(z) \le b,\\
-(2\pi) G(z, w)&\ge& \log \frac{(|z|-1)^2}{|z|^2+1}- \log 2 \ge -C_1 \quad \text{ if } d(z)>b
\end{eqnarray*}
where $C_1=\log \frac{(b-2)^2}{(b-1)^2+1}-\log 2$ is a constant independent of $z, w$. 
So we can use the formula \eqref{eq-potential} to get that there exists a constant $C=C(b)$ such that
\begin{eqnarray*}
p_\mu(z)&\ge& 2 \log d-C \quad \text{ if } d(z)\le b \\
p_\mu(z)&\ge& -C \quad \text{ if } d(z)>b.
\end{eqnarray*}
Since $d(z)$ is comparable to ${\rm dist}_{\rm sph}$ when $d(z)$ is small, we easily get the statement (iii).

\end{proof}
\begin{rem}\label{rem-semiexh}
Proposition \ref{prop-potential}.(iii) implies that $p_\mu$ is a semi-exhaustive function in the sense of \cite[III.(5.3)]{Dembook}).
\end{rem}

\section{Monge-Amp\`{e}re mass of $\vphi$ at the origin}
In this section, we will calculate the multiplicity of $\vphi$. As in section \ref{sec-mult}, consider the blowup $\rho: X\rightarrow \bC^2$ with exceptional divisor denoted by $D$. Over the coordinate chart $(U=\{z_2\neq 0\}, w=\frac{z_1}{z_2})$, we have:
\begin{equation}
\rho^*\vphi=\log|z_2|^2+\max\{\psi(w), \log|z_2|^2+2 \log(1+|w|^2)\}=\log |z_2|^2+\vphi'
\end{equation} 
where $\psi(w)=\log(1+|w|^2)+u(w)$. 
So we get an identity globally:
$$
\ddc \rho^*\vphi=[D]+\ddc\vphi'.
$$
%Now we define:
%\begin{equation*}
%(\ddc\vphi)^2=\ddc (\vphi \bo_{\bC^2\setminus \{0\}} (\ddc\vphi)).
%\end{equation*}
From now on, we set $u=p_\mu$ to be the potential for the probability measure on the generalized Cantor set from last section. In particular, $p_\mu$ satisfies the properties listed in Proposition \ref{prop-potential}. Note that the unbounded locus of $\vphi'$ is $\cC$ which is a compact set in $\bC^2$. So by \cite[III.(4.3)]{Dembook}, both $[D]\wedge \ddc\vphi'$ and $(\ddc\vphi')^2$ are well-defined, and are continuous in $\vphi'$ with respect to decreasing sequences.
\begin{lem}
We have the following formula:
\begin{equation}\label{eq-push}
(\ddc\vphi)^2=\rho_*\left([D]\wedge \ddc\vphi'+(\ddc\vphi')^2\right).
\end{equation}
As a consequence, we have the identity:
\begin{equation}\label{eq-mass0}
(\ddc\vphi)^2(\{0\})=1+(\ddc\vphi')^2(D).
\end{equation}
\end{lem}
\begin{proof}
Let $u_k$ be a sequence of smooth $\omega_0$-psh functions decreasing to $u$. Define
$\tilde{\vphi}_k=\max\{\log |z|^2+\pi^* u_k, 2\log |z|^2\}$. Then $\tilde{\vphi}_k$ decreases to $\vphi=\max\{\log |z|^2+\pi^* u, 2\log|z|^2\}$. 

Since both sides of \eqref{eq-push} are continuous with respect to decreasing sequences of psh functions. We just need to verify the identity for $\tilde{\vphi}_k$ for which the formula \eqref{eq-push} has been proved in \cite[Proposition 4.1]{AW14}. For the reader's convenience, we sketch the proof which is by induction. We change the notation and assume that $\vphi$ is one of $\tilde{\vphi}_k$ such that $\vphi'=\log(1+|w|^2)+u(w)$ is a locally bounded psh function. Then first have the identity:
$$
\ddc\vphi=\rho_*(\ddc \rho^*\vphi)=\rho_*([D]+\ddc \vphi').
$$ 
Next, we can verify the equality \eqref{eq-push} following \cite[Proof of Proposition 4.1]{AW14}:
\begin{eqnarray*}
(\ddc\vphi)^2&=&\ddc (\vphi \bo_{\bC^2\setminus \{0\}} (\ddc\vphi))=\ddc (\vphi \rho_*(\ddc\vphi'))=\rho_* \ddc((\rho^*\vphi) (\ddc\vphi'))\\
&=&\rho_*\ddc((\log|z_2|^2+\vphi') \ddc\vphi')=\rho_*([D]\wedge \ddc\vphi'+(\ddc\vphi')^2).
\end{eqnarray*}
Finally the identity \eqref{eq-mass0} holds true because $\int_D \ddc\vphi'=\int_D(\omega_0+\ddc u)=1$.
\end{proof}
%So we get:
%\begin{cor}
% Notations as above, we have the identity:
%\begin{equation}
%(\ddc\vphi)^2(\{0\})=1+(\ddc\vphi')^2(\cC).
%\end{equation}
%\end{cor}

\begin{lem}\label{lem-final}
There is an identity:
%\begin{equation}\label{eq-MAphi'}
%(\ddc\vphi')^2=[D]\wedge \ddc\psi.
%\end{equation}
%As a consequence, 
\begin{equation}\label{eq-Cmass}
(\ddc\vphi')^2(D)=(\ddc\vphi')^2(\cC)=\ddc\psi(\cC)=1.
\end{equation}
\end{lem}
This lemma completes the proof of Theorem \ref{thm-main}. Indeed, because $\mu=\omega_0+\ddc u$ does not have atoms, Corollary \ref{cor-Lelong} tells us that $\cJ(\lambda\vphi)=\cJ(\lambda\log|z|^2)$ for any $\lambda>0$ and $e_2(\vphi_m)= (\frac{m-1}{m})^2$. On the other hand, \eqref{eq-mass0} and \eqref{eq-Cmass} tells us that $e_2(\vphi)=(\ddc\vphi)^2(\{0\})=2$.

\begin{proof}[Proof of Lemma \ref{lem-final}]
%We first note that  \eqref{eq-MAphi'} implies \eqref{eq-Cmass}. 
%Indeed in our example constructed in earlier sections, 
Recall that $\vphi'(w, z_2)=\max\{\psi(w), \log|z_2|^2+\log(1+|w|^2)\}$ where 
$\psi(w)=\log(1+|w|^2)+u(w)$.  Because $\vphi'$ is locally bounded outside $\cC=\{\vphi'=-\infty\}$, by \cite[Corollary 3.3]{AW14} $(\ddc\vphi)^2$ does not charge any set contained in $D\setminus \cC \subset \bC^2$. This proves the first identity in \eqref{eq-Cmass}.

Next we set $\hat{\vphi}(w, z_2)=\max\{\psi(w), \log|z_2|^2\}$. Then because $\log(1+|w|^2)$ is bounded in a neighborhood of $\cC$, there exists a constant $C>0$ such that $\vphi'-C\le \hat{\vphi}\le \vphi'$. By Remark \ref{rem-semiexh} and the comparison theorem for generalized Lelong numbers in \cite[III.7.1]{Dembook}, we know that
$(\ddc\vphi')^2(\cC)=(\ddc\hat{\vphi})^2(\cC)$. So to prove the second identity of \eqref{eq-Cmass}, it suffices to prove the following identity:
\begin{equation}\label{eq-ddchphi}
(\ddc\hat{\vphi})^2=[D]\wedge \ddc\psi.
\end{equation}
Next by a formula of B\l ocki, we get:
\begin{equation*}
(\ddc \hat{\vphi})^2=\ddc\hat{\vphi}\wedge \ddc\psi+\ddc\hat{\vphi}\wedge \ddc\log|z_2|^2-\ddc\psi\wedge \ddc\log|z_2|^2.
\end{equation*}
This formula was proved in \cite[Theorem 4]{Blo00} when $\hat{\vphi}$ is given by a maximum of two locally bounded psh functions. It is still valid in our case because both sides are well-defined (the unbounded locs of $\hat{\vphi}$, i.e. $\cC$,  is a compact set in $\bC^2$) and are continuous with respect decreasing sequences.

Now we claim that each term on the right-hand-side is equal to $[D]\wedge \ddc\psi$, which will give us the identity \eqref{eq-ddchphi} thus finishing the proof.

The last term $\ddc\psi\wedge \ddc\log |z_2|^2=\ddc(\psi \log |z_2|^2)=\ddc\psi \wedge [D]$. We will deal with the first two terms using a same method, partly motivated by \cite[Proof of Theorem 7]{Blo00}.

\begin{enumerate}
\item
For any $\epsilon>0$ set $v_\epsilon=\log(|z_2|^2+\epsilon)$ and $\hat{\vphi}_\epsilon=\max\{\psi(w), v_\epsilon(z_2)\}$. Then $\hat{\vphi}_\epsilon\ge \log \epsilon$ and $\hat{\vphi}_\epsilon$ decreases to $\hat{\vphi}$ as $\epsilon \rightarrow 0+$. 
Pick any test function $\chi$. Using the fact that $\ddc\psi\equiv 0$ on $\{\psi>\log\epsilon\}$, we get:
\begin{eqnarray*}
\int \chi \ddc\hat{\vphi}_\epsilon\wedge \ddc \psi &=&\int \hat{\vphi}_\epsilon \ddc\chi\wedge \ddc \psi=\int_{\{\psi>\log\epsilon\}}\hat{\vphi}_\epsilon \ddc\chi\wedge \ddc\psi+
\int_{\{\psi\le \log\epsilon\}} \hat{\vphi}_\epsilon \ddc\chi\wedge \ddc\psi\\
&=&\int_{\{\psi\le \log\epsilon\}}v_\epsilon\ddc\chi \wedge \ddc\psi=\int v_\epsilon \ddc\chi \wedge \ddc\psi\\
&=&\int \chi \ddc v_\epsilon\wedge \ddc\psi.
\end{eqnarray*}
So we get $\ddc\hat{\vphi}_\epsilon\wedge \ddc\psi=\ddc v_\epsilon\wedge \ddc \psi$. Letting $\epsilon\rightarrow 0+$, we get $\ddc\hat{\vphi}\wedge \ddc\psi =\ddc\log|z_2|^2\wedge \ddc\psi=[D]\wedge \ddc\psi$ as wanted. 
\item To deal with the last term, we set $\psi_j=\max\{\psi, -j\}\ge -j$ and $\hat{\vphi}_j=\max\{\psi_j, \log|z_2|^2\}$ and use the same argument to get $\ddc\hat{\vphi}_j\wedge \ddc\log|z_2|^2= \ddc\psi_j\wedge [D]$ which converges to $\ddc\psi\wedge [D]$ as $j\rightarrow +\infty$. 

\end{enumerate}

\end{proof}

\section{Examples in higher dimensions and a question}\label{sec-hd}

Based on the above example, we can construct counter example to Conjecture \ref{conj-Dem} in higher dimensions. Choose the standard affine coordinate chart on $(\bC^{n-1}, \{z_1,\dots, z_{n-1}\})$ of $\bP^{n-1}$. Let $\mu$ be the Cantor probability measure on the generalized Cantor set $\cC$ which sits inside the $z_1$-line $\ell:=\{z_2=\cdots=z_{n-1}=0\}$ and let $p_\mu(z_1)$ be the potential given by \eqref{eq-potential}. Set 
\begin{equation}
\phi(z)=\max\{\log(1+|z_1|^2)+p_\mu(z_1), \log(|z_2|^2+\cdots+|z_{n-1}|^2)\}.
\end{equation}
Then $\phi(z)$ is a psh function on $\bC^n$ such that $\cC=\{\phi=-\infty\}$ and $\phi$ is continuous outside $\cC$. 
We can extend $\phi$ to be a psh metric on $\cO_{\bP^{n-1}}(\gamma)$ for $1\ll \gamma\in \bN$ as follows.
Denote by $A(r_1, r_2)=\{z\in \bC^{n-1}; r_1\le |z|\le r_2\}$ the standard closed spherical annulus centered at $0\in \bC^{n-1}$ such that $A(0, r)=B(r)=\{z\in \bC^{n-1}; |z|\le r\}$. Set $C_1=\min_{A(2,3)}\phi$ and $C_2=\max_{A(4,5)}\phi$. 
Choose $\gamma\gg 1$ and then $T\in \bR$ such that
\begin{equation}
C_2-\gamma \log (1+4^2)<T<C_1-\gamma \log (1+3^2).
\end{equation}
Set 
\begin{equation}
u(z)=\frac{1}{\gamma} \cdot \left\{
\begin{array}{ll}
\phi-\gamma \log(1+|z|^2) & z\in B(2) \\
\max\{\phi-\gamma \log(1+|z|^2), T\} & z\in A(2,5) \\
T & z\in \bP^{n-1}\setminus B(5).
\end{array}
\right.
\end{equation}
Then $u$ is $\omega_0$-psh function with zero Lelong number and satisfies $u(z)=\gamma^{-1}\cdot (\phi(z)-\gamma \log(1+|z|^2))$ in a neighborhood of $\cC\subset \bC^{n-1}$. 
Then as before we associate to $u$ a psh function
\begin{equation}\label{eq-u2phi}
\vphi=\max\{\log|z|^2+\pi^* u, 2\log|z|^2\}
\end{equation} 
which has an isolated singularity at $0\in \bC^{n}$. We leave it to the reader to use similar arguments as before to verify that $e_n(\vphi_m)=\left(\frac{m-n+1}{m}\right)^{n}$ and $e_{n}(\vphi)=1+\gamma^{-(n-1)}$. 

%It seems to be interesting to find similar examples which are maximal outside $0$. This probably could be constructed by an envelope construction. 
We can also construct counterexamples in neighborhoods of $0\in \bC^n$ that are maximal outside $\{0\}$, by using an argument in \cite[Lemma 2.1]{ACH19} as follows. Fix a bounded pseudoconvex domain $\Omega$ in $\bC^n$. Set
\begin{equation*}
\vphi^j={\rm sup}\left\{v\in {\rm PSH}(\Omega); v\le \vphi-\sup_{\Omega}\vphi \text{ on } B(0, j^{-1})^\circ, \text{ and } v\le 0\ \text{ on } \partial \Omega \right\}.
\end{equation*} 
Then $\{\vphi^j\}$ is an increasing sequence with limit denoted by $\tilde{\vphi}:=\lim_{j\rightarrow+\infty}\vphi^j$.  Note that $\vphi^j\ge \vphi-\sup_\Omega\vphi$. So the limit psh function $\tilde{\vphi}$ is less singular than $\vphi$, which implies $\cJ(m\tilde{\vphi})\supseteq \cJ(m\vphi)$. So $\tilde{\vphi}_{m}$ is also less singular than $\vphi_m$, which implies $e_n(\tilde{\vphi}_m)\le e_n(\vphi_m)$. By \cite[Lemma 2.1]{ACH19}, $\tilde{\vphi}$ satisfies $(\ddc\tilde{\vphi})^n=(\ddc\vphi)^n(\{0\})\delta_0$, whose proof we duplicate here for the reader's convenience. On the one hand, we have $(\ddc \vphi^j)^n=0$ on $\Omega\setminus B(0, j^{-1})$ and $(\ddc \vphi^j)^n=(\ddc \vphi)^n$ on $B(0, j^{-1})^\circ$. Letting $j\rightarrow+\infty$, we get $(\ddc\tilde{\vphi})^n=c \delta_{\{0\}}$ with $c\ge \int_{\{0\}}(\ddc \vphi)^n$. On the other hand, since $\tilde{\vphi}$ is less singular than $\vphi$, by the comparison theorem of Lelong numbers (\cite[III. \S 7]{Dembook}), we get $c\ge \int_{\{0\}}(\ddc \vphi)^n$. So indeed $c=\int_{\{0\}}(\ddc\vphi^n)$.  
In particular $e_n(\tilde{\vphi})=e_n(\vphi)>1$ when $\vphi$ comes from the previous examples.

We make the following conjecture, which could be thought as a local version of a result of Darvas from \cite{Dar17}.
\begin{conj}\label{conj}
Assume $u\in {\rm PSH}(\bP^{n-1}, \omega_0)$. Let $\vphi$ be defined as in \eqref{eq-u2phi}. Then $e_n(\vphi)=1$ if and only if $u \in \cE(\omega_0)$, where $\cE(\omega_0)$ denotes the set of $\omega_0$-psh functions of full mass as studied in \cite{GZ07}.  
\end{conj}
By using similar argument like before, one should be able to prove this when $e^u$ is continuous and $u=\{-\infty\}$ is contained in a compact subset of an affine chart. 
To establish this in general, it seems to require an appropriate definition of the Monge-Amp\`{e}re operator for general $u\in {\rm PSH}(\omega_0)$. 

We can also verify the above conjecture when $u$ has analytic singularities. In fact we will derive a formula for $e_n(\vphi)$ in this case, which seems to be interesting in its own right.
Recall that an $\omega_0$-psh function $u$ has analytic singularities if there is an ideal sheaf $\cI$ and $c\in \bR_{>0}$ such that for any point $x\in \bP^{n-1}$ there exists an open set $U$ such that $u-c \log \sum_i |f_i|^{2}$ is bounded on $U$ where $\{f_1,\dots, f_r\}$ are generators of $\cI(U)$. In this case, we say that the singularities of $\vphi$ are modeled on $\cI^c$. By the metric characterization of normalization (see \cite[Remark 9.6.10]{Laz04}), we can assume that the ideal sheaf is normal. Let $\rho_0: \hat{D}\rightarrow \bP^{n-1}$ be the (normalized) blowup of $\cI$ with exceptional divisor denoted by $E$. Denote by $H$ the hyperplane line bundle of $\bP^{n-1}$. Then the $\bR$-line bundle $\rho_0^*H-cE$ admits a bounded psh metric. For simplicity of notations, we will sometimes just use $H$ to denote the pull back of hyperplane line bundle under birational morphisms. 
\begin{prop}
Assume that $u\in {\rm PSH}(\omega_0)$ has analytic singularities and $\vphi$ is defined by the formula \eqref{eq-u2phi}.
With the above notations, we have the identity:
\begin{eqnarray}\label{eq-intersection}
e_n(\vphi)&=& 1+\sum_{p=0}^{n-2}(2^{n-1-p}-1)(H-cE)^p\cdot H^{n-2-p}\cdot cE.
%&=&(H-cE)^{n-1}+(2H)^{n-1}-.
\end{eqnarray}
%Moreover, the identity holds true if $\rho_0^*H-cE$ is a semiample $\bQ$-divisor. 
\end{prop}
%\begin{rem}
%We expect that \eqref{eq-intersection} is always an identity.  
%\end{rem}
\begin{proof}
We first deal with the case when $A:=\rho_0^*H-cE$ is a semiample $\bQ$-divisor. The proof in this case is reduced to a purely algbro-geometric calculation by the following construction. 
Choose $d \gg 1$ such that $cd\in \bN$ and $\cO_{\hat{D}}(dA)$ is globally generated. 
Let $\{s_1,\dots, s_N\}$ be a basis of $H^0(\hat{D}, \cO_{\hat{D}}(dA))$ and $f_i=s_i\cdot (dc) s_E$ be the corresponding homogeneous polynomial of degree $d$. Consider the psh function:
\begin{equation}
\phi_d(z)=\max\left\{\log \sum_i |f_i|^2, 2 d \log|z|^2 \right\}
\end{equation}
Set $\fm=(z_1,\dots, z_n)$ and $\fa=\la f_i; i=1,\dots, N\ra+\fm^{2d}$. Then by the comparison theorem for generalized Lelong numbers (\cite[III.7.1]{Dembook}) and \cite[Lemma 2.1]{Dem09}, we get the identity:
\begin{equation}
e_n(\vphi)=e_n(d^{-1}\phi_d)=d^{-n}\mult(\fa)
\end{equation}
where $\mult(\fa)$ is the Hilbert-Samuel multiplicity of the primary ideal $\fa$ which we now calculate using intersection theory. 
%where $\{q_j\}$ is a basis for the set of homogeneous polynomials of degree $2 d$. intersection theory to calculate $\mult(\fa)$.

Let $\rho_1: X\rightarrow \bC^n$ be the standard blowup of the origin with exceptional divisor denoted by $D$ which is isomorphic to $\bP^{n-1}$. We get the identity $\rho_1^*\fa=\mcO(-d D)\cdot (\cI'+\cI_D^d)$ where $\cI'=\cI^{cd}$.  As mentioned above, we assume that $\cI$ is normal. Let $\rho_2: \hat{X}\rightarrow X$ be the (normalized) blowup of $\cI'+\cI_D^d$. Then $\hat{X}$ is given by: 
\begin{equation*}
\hat{X}={\rm Proj}\left(\bigoplus_{m=0}^{+\infty} {(\cI'+\cI_D^d)^m}\right).
\end{equation*}
Note that ${(\cI'+\cI_D^d)^m}=\sum_{k=0}^{m}{\cI'^k}\cdot \cI_D^{d(m-k)}$. 
The exceptional divisor is given by:
\begin{equation*}
\cE'={\rm Proj}\left(\bigoplus_{m=0}^{+\infty} \sum_{k=0}^{m} {\cI'^k}/{\cI'^{k+1}}\cdot \cI_D^{d(m-k)}\right).
\end{equation*}
Let $\tilde{\xi}$ be the line bundle $\cO(-\cE')$. Then $\tilde{\xi}|_{\cE'}=\cO_{\cE'}(1)$. 

On the other hand, let $\rho_0: \hat{D}\rightarrow D$ be the normalized blowup of $D\cong\bP^{n-1}$ along $Z$, which is also strict transform of $D$ under $\rho_2$. Then $\hat{D}$ is given by:
\begin{equation*}
\hat{D}={\rm Proj}\left(\bigoplus_{m=0}^{+\infty}{\cI'^m} \right)
\end{equation*}
with exceptional divisor given by:
$$
E'={\rm Proj}\left(\bigoplus_{m=0}^{+\infty} {\cI'^m}/{\cI'^{m+1}} \right).
$$ 
Let $\xi$ be the line bundle $\cO(-E')$. Then $\xi|_{E'}=\cO_{E'}(1)$.

Let $\tilde{\xi}'=\tilde{\xi}\otimes \rho_2^*H^{-d}$. Then $\tilde{\xi}'|_{\cE'}\sim \cO_{\cE'}(E')$. So we get $\tilde{\xi}|_{\cE'}=\tilde{\xi}'\otimes \rho_2^*H^d|_{\cE'}=\cO_{\cE'}(E')\otimes \pi^*H^d$.

Set $\rho=\rho_2\circ \rho_1$. Then $\rho^*\fa=\rho_2^*\cO_X(-d D)\cdot \cO_{\hat{X}}(-\cE')$ and the multiplicity of $\fa$ is given by:
\begin{equation}\label{eq-multint}
\left(-\rho_2^*(dD)-\cE'\right)^{n-1} \cdot (\rho_2^*(dD)+\cE')=(\rho_2^*(dH)+\tilde{\xi})^{n-1}\cdot (\rho_2^*(dD)+\cE').
\end{equation}

Set $\rho_2^*(dH)=x$ and $a_k=(x+\tilde{\xi})^k\cdot x^{n-1-k}\cdot \rho_2^*(dD)$. Then 
\begin{eqnarray*}
a_k&=&(x+\tilde{\xi})^{k-1}\cdot x^{n-k} \cdot \rho_2^*(dD)+(x+\tilde{\xi})^{k-1}\cdot \tilde{\xi}\cdot x^{n-1-k} \cdot \rho_2^*(dD)\\
&=&a_{k-1}+(x+\tilde{\xi})^{k-1}\cdot x^{n-1-k}\cdot \rho_2^*(dD) \cdot (-\cE')=a_{k-1}+b_{k-1}
\end{eqnarray*}
where $b_k=(x+\tilde{\xi})^k\cdot x^{n-1-k}\cdot \cE'$. 
Note that $a_0=x^{n-1}\cdot \rho_2^*(dD)=d^n H^{n-1}\cdot \bP^{n-1}=d^n$. 

On the other hand, we have: 
\begin{eqnarray*}
b_k&=&(x+\tilde{\xi})^{k-1}\cdot x^{n-k}\cdot \cE'+(x+\tilde{\xi})^{k-1}\cdot x^{n-1-k} \cdot (\tilde{\xi}'+x) \cdot \cE'\\
&=&2 (x+\tilde{\xi})^{k-1}\cdot x^{n-k}\cdot  \cE'+(dH-E')^{k-1}\cdot (dH)^{n-1-k}\cdot E'\\
&=&2 b_{k-1}+ (dH-E')^{k-1}\cdot H^{n-1-k}\cdot E'.
\end{eqnarray*}
Since $b_0=0$, by induction we get:
\begin{eqnarray*}
b_k&=&\sum_{j=1}^{k} 2^{j-1} (dH-E')^{k-j}\cdot (dH)^{n-2-k+j}\cdot E'\\
&=&d^n\sum_{j=1}^k 2^{j-1} (H-c E)^{k-j}\cdot H^{n-2-k+j}\cdot (cE).
\end{eqnarray*}
So we get the formula for $\mult(\fa)$:
\begin{eqnarray*}
\mult(\fa)&=&a_{n-1}+b_{n-1}=H^n+\sum_{k=1}^{n-2} b_k+b_{n-1}=d^n+\delta
\end{eqnarray*}
where we have:
\begin{eqnarray*}
d^{-n} \delta&=&\sum_{k=1}^{n-1} d^{-n} b_k=\sum_{k=1}^{n-1}\sum_{j=1}^{k}2^{j-1}(H-cE)^{k-j}\cdot H^{n-2-k+j}\cdot (cE)\\
&=&\sum_{p=0}^{n-2} \sum_{j=1}^{n-p} 2^{j-1} (H-cE)^p\cdot H^{n-2-p}\cdot (cE)\\
&=&\sum_{p=0}^{n-2} (2^{n-1-p}-1)(H-cE)^p\cdot H^{n-2-p}\cdot cE.
\end{eqnarray*}
Combining the above identities, we get the formula \eqref{eq-intersection} when $c\in \bQ_{>0}$.  

Now we consider the general case. The claim is that the above algebraic calculation applies equally well to the general case, i.e. to general $c\in \bR_{>0}$. 
As in the proof of \eqref{eq-push}, this claim follows from the work \cite{AW14}. Indeed, if we use the same blowup process $\rho=\rho_2\circ \rho_1: \hat{X}\rightarrow \bC^n$ as above with $\rho_2^*\cI=\cO_{\hat{X}}(-\cE)$, then we have the identity:
\begin{equation*}
\rho^*\vphi=\log|F|^2+\vphi'=\log|F_1|^2+\log |F_2|^{2c}+\vphi'
\end{equation*} 
where $|F|^2=|F_1|^2 |F_2|^{2c}$ with $F_1$ (resp. $F_2$) being a locally defined holomorphic functions on $\hat{X}$ defining the Cartier divisor $\rho_2^*D$ (resp. $\cE$), and $\vphi'$ is a bounded psh function locally defined up to a pluriharmonic function. $\ddc\vphi'$ is globally defined and represents the curvature of the $\bR$-line bundle associated to the $\bR$-Cartier divisor $-\rho_2^*D-c \cE$. By the same argument as in \cite[Proof of (4.5)]{AW14}, we can calculate as follows:
\begin{eqnarray*}
(\ddc\vphi)^n&=&\rho_*\ddc((\rho^*\vphi)({\bf 1}_{\hat{X}\setminus {\rm Supp}(\rho_2^*D\cup \cE)}(\ddc\vphi)^{n-1})\\
&=&\rho_* \ddc ((\log|F|^2+\vphi')(\ddc\vphi')^{n-1})\\
&=& \rho_* (([\rho_2^*D]+c[\cE])\wedge (\ddc\vphi')^{n-1}+(\ddc\vphi')^n)
\end{eqnarray*}
So we get:
\begin{equation*}
(\ddc\vphi)^n(\{0\})=\int_{\rho_2^*D+c \cE}(\ddc\vphi')^{n-1}=(-\rho_2^*D-c \cE)^{n-1}\cdot (\rho_2^*D+\cE).
\end{equation*}
Now the same method for calculating \eqref{eq-multint} applies without change to any $c\in \bR_{>0}$, and hence we have proved our formula \eqref{eq-intersection}.  
%First we know that there exists $\epsilon\ll c$ such that $\rho_0^*H-\epsilon E$ is ample on $\hat{D}$. 
%So for any $t>0$, $(1-t)(H-cE)+t (H-\epsilon E)=H-((1-t)c+t\epsilon)E$ is ample and hence semiample. 
%Fix an increasing sequence of rational numbers $\{c_k\}\subset \bQ_{>0}$ that converges to $c$. 
%Note that
%\begin{equation*}
%A_k:=\rho_0^*H-c_k E=(1-\frac{c-c_k}{c-\epsilon})(\rho_0^*H-cE)+\frac{c-c_k}{c-\epsilon}(\rho_0^*H-\epsilon E)
%\end{equation*}
%is a semiample $\bQ$-divisor. 
%Let $u_k$ be a $\omega_0$-psh metric on $H$ induced by the identity $$\vphi_{A_k}+c_k\log |s_E|^2-\vphi_H=\frac{c_k}{c}\vphi_A+(1-\frac{c_k}{c})\vphi_H+c_k\log|s_E|^2-\vphi_H$$
%where $\vphi_A$ and $\vphi_H$ denote fixed smooth psh metrics on $A$ and $H$ respectively. Then $u_k$ is less singular than $u$. If we use $u_k$ to define $\vphi_k$ using formula \eqref{eq-u2phi}, then $\vphi_k$ is less singular than $\vphi$. This implies $e_n(\vphi)\ge e_n(\vphi_k)$. Since the right-hand-side of \eqref{eq-intersection} is continuous in $c$, we let $k\rightarrow +\infty$ to get $e_n(\vphi)$ is indeed bounded below by the right-hand-side of \eqref{eq-intersection}. 
%To get the other direction, we choose a sequence of decreasing sequence of rational numbers $\{c'_k\}\subset \bQ_{>0}$ that converges to $c$. 
\end{proof}
\begin{exmp}
Assume that $\cI=\cI_{Z}$ where $Z=\{z_1=0\}\subset \bP^{n-1}$. For any $c\in [0,1]$, we get an $\omega_0$-psh function $u$ with analytic singularities modeled on $\cI^c$. We then have the identity:
\begin{eqnarray*}
e_n(\vphi)&=&1+\sum_{k=0}^{n-2}(2^{n-1-k}-1)(H-cH)^k\cdot H^{n-2-k}\cdot c H\\
&=&1+c\cdot \sum_{k=0}^{n-2}(2^{n-1-k}-1)(1-c)^k
\end{eqnarray*}
When $c=\frac{p}{q}$ with $p\le q$ and $p, q$ relatively prime, we can choose $d=q$ to get a monomial ideal $\fa=\la z_1^p\ra\cdot\fm^{q-p}+\fm^{2q}$.
Using Newton polytope associated to $\fa$, it is easy to calculate 
%$$\mult(\fa)=q^3+4p q^2-p^2q.$$ 
$$\mult(\fa)=(q-p)^n+p \sum_{k=0}^{n-1}(q-p)^k\cdot (2q)^{n-1-k}.$$
Moreover it is easy to verify the identity:
%\begin{equation}
%e_3(\vphi)=2^{-3}\mult(\fa)=1+(4-1) H\cdot (\frac{1}{2}H)+(2-1) (H-\frac{1}{2}H)\cdot \frac{1}{2}H=\frac{11}{4}.
%\end{equation}
\begin{eqnarray*}
e_n(\vphi)%&=&1+\sum_{k=0}^{n-2}(2^{n-1-k}-1)(H-\frac{p}{q}H)^k\cdot H^{n-2-k}\cdot \frac{p}{q}H\\
&=&q^{-n}\left(q^n-p\sum_{k=0}^{n-1}(q-p)^k\cdot q^{n-1-k}+p\cdot \sum_{k=0}^{n-1}(q-p)^k \cdot (2q)^{n-1-k}\right)\\
&=&q^{-n}\mult(\fa).
\end{eqnarray*}
%\begin{eqnarray*}
%e_3(\vphi)&=&q^{-3}\mult(\fa)=1+(4-1)H\cdot \frac{p}{q}H+(2-1)(H-\frac{p}{q}H)\cdot \frac{p}{q}H\\
%&=&1+\frac{4p}{q}-\frac{p^2}{q^2}=1+4c-c^2.
%\end{eqnarray*}
\end{exmp}
\begin{exmp}
Assume $\cI=\cI_p$ where $p=[0,\cdots, 0, 1]$. For any $c\in [0,1]$, we get an $\omega_0$-psh function $u$ with analytic singularities modeled on $\cI^c$. We then have the identity:
\begin{eqnarray*}
e_n(\vphi)&=&1+\sum_{k=0}^{n-2}(2^{n-1-k}-1)(H-cE)^k\cdot H^{n-2-k}\cdot cE\\
&=&1+c^{n-1}.
\end{eqnarray*}
When $c=p/q$ with $p\le q$ and $p,q$ relatively prime, we can choose $d=q$ to get a monomial ideal $\fa=\la z_1, \dots, z_{n-1}\ra^p\cdot \fm^{q-p}+\fm^{2q}$. Using the Newton polytope associated to $\fa$, it is easy to calculate %$\mult(\fa)=q^3+p^2q$.
$$
\mult(\fa)=p^{n-1}(q+p)+(q-p)\cdot \sum_{k=0}^{n-1} p^k \cdot q^{n-1-k}=p^{n-1}q+q^n.
$$
%Let $n=2$, $c=1/3$ and $\cI=\cI_{p}$ where $p=[0,0,1]$. Choosing $d=3$, we get a monomial ideal
%$\fa=\la z_1^3, z_2^3, z_1^2z_2, z_1^2z_3, z_2^2z_1, z_2^2z_3, z_3^2z_1, z_3^2z_2, z_1z_2z_3, q_j; 1\le j\le 28\ra$ where $\{q_j\}$ is a basis for the set of homogeneous monomials of degree $6$. Using the polytope associated to $\fa$, it is easy to calculate $\mult(\fa)=30$. Then
%\begin{equation*}
%e_3(\vphi)=3^{-3}\mult(\fa)=1+(4-1)\rho_0^*H\cdot \frac{1}{3}E+(2-1)(\rho_0^*H-\frac{1}{3}E)\cdot \frac{1}{3}E=\frac{10}{9}.
%\end{equation*}
Again it is immediate to verify the identity:
\begin{eqnarray*}
e_n(\vphi)&=&1+\frac{p^{n-1}}{q^{n-1}}=q^{-n}\mult(\fa).
\end{eqnarray*}
%\begin{eqnarray*}
%e_3(\vphi)&=&q^{-3}\mult(\fa)=1+(4-1)\rho_0^*H\cdot \frac{p}{q}E+(2-1)(\rho_0^*H-\frac{p}{q}E)\cdot \frac{p}{q}E\\
%&=&1+\frac{p^2}{q^2}=1+c^2.
%\end{eqnarray*}
\end{exmp}

\begin{cor}
Assume that $u\in {\PSH}(\omega_0)$ has analytic singularities.
Then $e_n(\vphi)=1$ if and only if $\cI=\cO_{\bP^{n-1}}$, if and only if  $u\in \cE(\omega_0)$. 
\end{cor}
\begin{proof}
If $\cI=\cO_{\bP^{n-1}}$ then $\vphi$ has the same type of singularity as $\log|z|^2$ and hence $e_n(\vphi)=1$. 
Formula \eqref{eq-intersection} implies that $e_n(\vphi)=1$ only if 
\begin{equation}\label{eq-vanish}
(H-cE)^p \cdot H^{n-2-p}\cdot cE=0, \text{ for all } p=0,\dots, n-2.
\end{equation}
Note that \begin{equation*}
\delta:=H^{n-1}-(\rho_0^*H-cE)^{n-1}=\sum_{p=0}^{n-2} (H-cE)^p\cdot H^{n-2-p} \cdot (cE).
\end{equation*}
On the other hand, the non-pluripolar volume of $\omega_0+\ddc u$ is equal to $(\rho_0^*H-cE)^{n-1}$. 
So $u\in \cE(\omega_0)$ if and only if $\delta=0$, if and only if the condition \eqref{eq-vanish} holds true. 

Note that 
$$
\delta=H^{n-1}-(H-cE)^{n-1}=(n-1) \int_0^c (H-t E)^{n-2}\cdot E dt.
$$
Because $H-\epsilon E$ is ample for $0<\epsilon \ll 1$, $\delta=0$ if and only if $E=\emptyset$, i.e. $\cI=\cO_{\bP^{n-1}}$.

% if and only if $E=\emptyset$.

\end{proof}
%Finally, it would be interesting to find a maximal $\vphi$ that violates the convergence. 
%\begin{equation}
%\cJ(m\vphi)=\{f\in \mcO(\bB); \int_{\bB} |f|^2e^{-m\vphi}d\lambda<+\infty\}.
%\end{equation}
%Let $\{f^{(m)}_i\}$ be an orthonormal basis 

%More concretely we consider the following diagram:
% \begin{equation}\label{eq-common}
%\xymatrix{% @R=1.5pc @C=0.5pc{
%& \ar_{p_1}[ld] \mcZ \ar^{p_2}[d]   \ar^{p_3}[rd] & \\
%\mcX  \ar@{-->}[r]  \ar_{\pi_{\mcX}}[rd] &  X_\bC   & \ar^{\pi_\mcY}[ld] \mcY \ar@{-->}[l]   \\
% & \bC  & %\\
%Z_\bC=Z^{(1)}_\bC \ar^{\bar{\sigma}_\xi}@{-->}[rr]  &  & Z_\bC=Z^{(2)}_\bC 
%}
%\end{equation}

%\appendix

%\begin{eqnarray*}
%\bfM_\psi(\vphi)&=&\int_0^1 dt \int_X\dot{\vphi} n(-Ric((\ddc\vphi))+Ric(\Omega))-n Ric(\Omega)+ \udS (\ddc\vphi))\wedge (\ddc\vphi)^{n-1}\\
%&=&\int_0^1 dt\; n \int_X \dot{\vphi}\left(\ddc\log\frac{(\ddc\vphi)^n}{\Omega}\right)\wedge (\ddc\vphi)^{n-1}-n E^{Ric(\Omega)}_\psi(\vphi)+\udS E_\psi(\vphi)\\
%&=&\int_X \log \frac{(\ddc\vphi)^n}{\Omega}(\ddc\vphi)^n-\int_X \log\frac{(\ddc\psi)^n}{\Omega}(\ddc\psi)^n-n E^{Ric(\Omega)}_\psi(\vphi)+\udS E_\psi(\vphi)\\
%&=&\int_X \log\frac{(\ddc\vphi)^n}{(\ddc\psi)^n}(\ddc\vphi)^n-n E^{Ric(\psi)}_\psi(\vphi)+\udS E_\psi(\vphi).
%\end{eqnarray*}

\vskip 3mm

\noindent
Department of Mathematics, Purdue University, West Lafayette, IN, 47907-2067.

\noindent
{\it Current address:} Department of Mathematics, Rutgers University, Piscataway, NJ 08854-8019.

\noindent
{\it E-mail address:} chi.li@rutgers.edu

\vskip 2mm

%\noindent
%School of Mathematical Sciences and BICMR, Peking University, Yiheyuan Road 5, Beijing, P.R.China, 100871

%\noindent {\it E-mail address:} tian@math.princeton.edu

%\vskip 2mm

%\noindent
%School of Mathematical Sciences, Zhejiang University, Zheda Road 38, Hangzhou, Zhejiang,
%310027, P.R. China

%\noindent {\it E-mail address:} wfmath@zju.edu.cn

\end{document}